\documentclass[hidelinks,12pt, reqno]{amsart}
\usepackage{amsmath}
\usepackage{amssymb,amscd}
\usepackage{graphicx}
\usepackage{extarrows}
\usepackage{latexsym} 
\usepackage{tikz-cd}
\usepackage{enumitem}
\usepackage{float}
\usepackage{hyperref}
\usepackage{verbatim} 

\usepackage{xpatch}
\makeatletter   
\xpatchcmd{\@tocline}
{\hfil\hbox to\@pnumwidth{\@tocpagenum{#7}}\par}
{\ifnum#1<0\hfill\else\dotfill\fi\hbox to\@pnumwidth{\@tocpagenum{#7}}\par}
{}{}
\makeatother 

\makeatletter
 \def\l@subsection{\@tocline{2}{0pt}{4pc}{6pc}{}}
\def\l@subsubsection{\@tocline{3}{0pt}{8pc}{8pc}{}}
 \makeatother

\numberwithin{equation}{section}

\newcommand{\C}{\mathbb{C}}

\newcommand{\R}{\mathbb{R}}
\newcommand{\Z}{\mathbb{Z}}
\newcommand{\T}{\mathbb{T}}

\DeclareMathOperator{\pr}{Pr}
\DeclareMathOperator{\id}{Id}

\DeclareMathOperator{\ann}{Ann}
\DeclareMathOperator{\img}{img}

\newtheorem{theorem}{Theorem}[section]
\newtheorem{corollary}[theorem]{Corollary}
\newtheorem{lemma}[theorem]{Lemma}
\newtheorem{prop}[theorem]{Proposition}

\theoremstyle{definition}
\newtheorem{definition}[theorem]{Definition}
\newtheorem{remark}[theorem]{Remark}
\newtheorem{example}[theorem]{Example}

\begin{document}

\title[Generalized complex structure on principal torus bundles ]{Generalized complex structure on certain principal torus bundles}
\author[D. Pal]{Debjit Pal}

\address{Department of Mathematics, Indian Institute of Science Education and Research, Pune, India}

\email{debjit.pal@students.iiserpune.ac.in; mathdebjit@gmail.com}

\author[M. Poddar]{Mainak Poddar}

\address{Department of Mathematics, Indian Institute of Science Education and Research, Pune, India}

\email{mainak@iiserpune.ac.in}

\subjclass[2020]{Primary: 53D18, 57R22, 57R30.}

\keywords{ Generalized complex structure, generalized Dolbeault cohomology,  generalized Darboux theorem, principal bundles.}

\begin{abstract} A principal torus bundle over a complex manifold with even dimensional fiber and characteristic class of type $(1,1)$ admits a family of regular generalized complex structures (GCS) with the fibers as leaves of the associated symplectic foliation. We show that such a generalized complex structure is equivalent to the product of the complex structure on the base and the symplectic structure on the fiber in a tubular neighborhood of an arbitrary fiber if and only if the bundle is flat. This has consequences for the generalized Dolbeault cohomology of the bundle that includes a K\"{u}nneth formula. On a more general note, if a principal bundle over a complex manifold with a symplectic structure group admits a GCS with the fibers of the bundle as leaves of the associated symplectic foliation, and the GCS is equivalent to a product GCS in a neighborhood of every fiber, then the bundle is flat and symplectic.    \end{abstract}

\maketitle

\renewcommand\contentsname{\vspace{-1cm}}
\tableofcontents

\section{Introduction}\label{intro}

Generalized complex (GC) geometry presents a unified framework for a range of geometric structures whose two extreme cases are complex and symplectic structures. The notion was introduced by Hitchin \cite{Hit} and developed to a large extent by his doctoral students Gualtieri \cite{Gua, Gua2} and Cavalcanti \cite{Cavth}. A generalized complex structure (GCS) induces a possibly singular foliation with symplectic leaves.  A point on a GC manifold is called regular if the dimension of the symplectic leaf is constant near it. In the neighborhood of a regular point, the GCS induces a complex structure on the leaf space of the symplectic foliation. Moreover, by a generalized Darboux theorem due to Gualtieri, the GCS is equivalent to the product of a leaf-wise symplectic structure and a transverse complex structure in a neighborhood of a regular point.

Although there are several examples of generalized complex structures that are neither symplectic nor complex, they require some effort to construct.  In his thesis \cite{Cavth},
Cavalcanti observed the existence of a family of generalized complex structures on an even dimensional torus principal bundle over a complex manifold with characteristic class of type $(1,1)$. The symplectic foliation in this case is regular and the leaves are the torus fibers of the bundle. The transverse complex structure on the leaf space coincides with the complex structure on the base manifold. In this article, we study this family more closely. By the Darboux theorem, and the fact that the symplectic structure is invariant along the torus fibers, it is tempting to speculate that the GCS may admit a product description in a tubular neighborhood of an entire fiber.  On the other hand, the proof of the Darboux theorem indicates that the cohomology of the symplectic leaf, i.e., the torus fiber, may be an obstruction to achieving such a product description.

We prove the following result (Theorem \ref{thm:productgcs}): In a trivializing neighborhood of a torus fiber, any GCS belonging to the above family is equivalent to the product of the symplectic structure on the fiber and the complex structure on the base up to diffeomorphisms and $B$-transforms if and only if the principal bundle is flat. In fact, Theorem \ref{thm:productgcs3} shows that if a principal $G$-bundle, where $G$ is a Lie group with a symplectic structure, admits a GCS which is locally equivalent to a product GCS in a neighborhood of every fiber and the symplectic leaves of the GCS are the fibers of the bundle, then the bundle is flat and symplectic.
Using this, we deduce a stronger version of Theorem \ref{thm:productgcs}, namely, Theorem \ref{thm:productgcs2}, which
says that a principal torus bundle over a complex manifold is symplectic and flat if and only if it admits a GCS which is equivalent to a product GCS in a neighborhood of each torus fiber.  

An application of Theorem \ref{thm:productgcs} is that the spectral sequence developed by Angella et al. \cite{angella} can be applied to describe the generalized Dolbeault cohomology of the total space of the bundle. This is explained in the more general setting of symplectic fiber bundles with suitable assumptions on the GCS that are slightly more general than the hypotheses of \cite{angella} (see Theorems \ref{main2} and \ref{main3}). The case of principal torus bundles is stated in Corollary \ref{main4}, and a K\"{u}nneth formula for the generalized Dolbeault cohomology of these bundles is given in Corollary \ref{main5}.     

It may be noted that there are many natural classes of even dimensional manifolds that admit the structure of a torus principal bundle over a complex manifold with characteristic class of type $(1,1)$. These include the 
product of two odd dimensional spheres (more generally,  a large class of moment angle manifolds in toric topology), even dimensional compact connected Lie groups, and total spaces of unitary frame bundles associated with holomorphic vector bundles of even rank, etc. (see Theorem \ref{cxtion} or \cite{PT}.)

\section{Preliminaries}\label{prelim}

We start by recalling the setup of generalized complex geometry. In this section, we rely upon  \cite{Gua} and \cite{Cav05} for most of the definitions and results. To define a GCS on an even dimensional smooth manifold $M$, we need three key ingredients.
Firstly,  given any $2n$-dimensional smooth manifold $M$, the direct sum of the tangent and cotangent bundles of $M$, which we denote by $T\oplus T^{*}$, is endowed with a natural symmetric bilinear form of signature $(2n,2n)$
\begin{equation}\label{bilinear}
    \langle X+\xi,Y+\eta\rangle\,:=\,\frac{1}{2}(\xi(Y)+\eta(X))\,.
\end{equation}
Secondly, we need the \textit{Courant Bracket} on the smooth sections of $T\oplus T^{*}$ which is defined as follows. 
\begin{definition}
The Courant bracket is a skew-symmetric bracket defined on smooth sections of $T\oplus T^{*}$, given by
\begin{equation}\label{bracket}
    [X+\xi,Y+\eta] := [X,Y]_{Lie}+\mathcal{L}_{X}\eta-\mathcal{L}_{Y}\xi-\frac{1}{2}d(i_{X}\eta-i_{Y}\xi),
\end{equation}  
where $X,Y\in\Gamma(T)$, $\xi,\eta\in\Gamma(T^{*})$, $[\,,\,]_{Lie}$ is the usual Lie bracket on vector fields, and $\mathcal{L}_{X},\,\, i_{X}$ are the Lie derivative and the interior product of forms with respect to the vector field $X$, respectively.
\end{definition}
For the third ingredient, consider the action of $T\oplus T^{*}$  on $\wedge^{\bullet}T^{*}$ defined by $$(X+\xi)\cdot\varphi = i_X\varphi+\xi\wedge\varphi \,.$$  This action can be extended to the Clifford algebra of $T\oplus T^{*}$ corresponding to the natural pairing \eqref{bilinear}. This gives a natural choice for spinors, namely, the exterior algebra of cotangent bundle, $\wedge^{\bullet}T^{*}$.
Define a linear map $\alpha$ on $\wedge^{\bullet}T^{*}$ which acts on decomposable forms by
$$\alpha(a_1\wedge\ldots\wedge a_{i})=a_{i}\wedge\ldots\wedge a_1 \, .$$
\begin{definition}
Given two forms of mixed degree $\sigma_{i}=\sum\sigma^{k}_{i}$, $i=1,2$, where $\deg(\sigma^{k}_{i})=k$, in an $n$-dimensional vector space, we define their pairing, $(\sigma_{1}\,,\,\sigma_{2})$ by
\begin{equation}\label{mukai}
    (\sigma_{1}\,,\,\sigma_{2})=(\alpha(\sigma_{1})\wedge\sigma_{2})_{Top},
\end{equation} where $Top$ indicates the degree $n$ component of the wedge product.
\end{definition}

Now, we are ready to present the notion of the generalized complex structure (GCS) on a $2n$-dimensional smooth manifold $M$ in three equivalent ways.
\begin{definition}\label{gcs}
A \textit{generalized complex structure} or {\it GCS} is determined by any of the following three equivalent sets of data:
\begin{enumerate}
\setlength\itemsep{1em}
     \item A subbundle $L$ of $(T\oplus T^{*})\otimes\C$ which is maximal isotropic with respect to the natural bilinear form \eqref{bilinear}, and involutive with respect to the Courant bracket \eqref{bracket}, and also satisfies $L\cap\Bar{L}=\{0\}$.
    \item  A bundle automorphism $\mathcal{J}$ of $T\oplus T^{*}$ which satisfies the following conditions:
    \begin{itemize}
    \setlength\itemsep{1em}
        \item[(a)] $\mathcal{J}^{2}=-1$
        \item[(b)] $\mathcal{J}^{*}=-\mathcal{J}$, i.e., $\mathcal{J}$ is orthogonal with respect to the natural pairing \eqref{bilinear}
        \item[(c)] $\mathcal{J}$ has vanishing {\it Nijenhuis tensor}, i.e., 
   $$ N(A, B) := -[\mathcal{J}A, \mathcal{J}B] + \mathcal{J} [\mathcal{J}A, B] + \mathcal{J} [A, \mathcal{J} B] 
    + [A, B] =0 $$
    for all $A, B \in\Gamma(T\oplus T^{*})$. 
    \end{itemize}
  
    \item A line subbundle $U$ of $\wedge^{\bullet}T^{*}\otimes\C$ which generated locally at each point by a form of the form $\rho=e^{(B+i\omega)}\wedge\Omega$, such that the pairing \eqref{mukai} $$(\rho\,,\,\Bar{\rho})=\omega^{n-k}\wedge\Omega\wedge\overline{\Omega}\neq 0,$$ where $B$ and $\omega$ are real 2-forms and $\Omega$ is a decomposable complex $k$-form, and  $\rho$ satisfies
    \begin{equation}\label{d rho}
        d\rho=u\cdot\rho,
    \end{equation}
   for some $u\in (T\oplus T^{*})\otimes\C$, where $d$ is the exterior derivative. 
\end{enumerate}

\end{definition}
At each point, the degree of $\Omega$ is called the \textit{type} of the GCS  at that point.  A point near which the type is locally constant is called a \textit{regular point}. If every point is regular, we say that the GCS is regular. The line bundle $U$ that defines the GCS  is called the \textit{canonical line bundle}. 

Given a GCS  $\mathcal{J}$ on a $2n$-dimensional manifold $M$, we get a decomposition of the complex of differential forms as follows: Let $U\subset\wedge^{\bullet}T^{*}\otimes\C$ be the canonical line bundle of $\mathcal{J}$. Then the $+i$-eigenbundle $L$ of $\mathcal{J}$ in $(T\oplus T^{*})\otimes\C$ may be obtained as $$L=\ann(U)=\{ u\in (T\oplus T^{*})\otimes\C\,|\,u\cdot U=0\} \,.$$ 
For each $i\in\Z$,  define $$U^{i}:=\wedge^{n-i}\Bar{L}\cdot U\,\,\subset\wedge^{\bullet}T^{*}\otimes\C.$$
Note that $U^{i}=0$ for each $i<-n$ and $i>n$, and $U^{n}$ is the canonical line bundle $U$. We have $$\wedge^{\bullet}T^{*}\otimes\C=\bigoplus^{n}_{i=-n}U^{i} \,.$$
Denote by $\Gamma(U^{i})$ the  vector space of smooth sections of $U^{i}$. Then \cite[Theorem 4.23]{Gua} implies that
\begin{equation}\label{d eq}
    d: \Gamma(U^{i})\longrightarrow\Gamma (U^{i+1})\oplus\Gamma(U^{i-1}).
\end{equation}
decomposes into two operators as $d=\partial+\Bar{\partial}$. The $\Bar{\partial}$ and $\partial$ operators are defined by composing $d$ with the projections onto
$ \Gamma(U^{i-1})$ and $\Gamma (U^{i+1})$, respectively,
\begin{equation}\label{Bar partial}
    \Bar{\partial}: \Gamma(U^{i})\longrightarrow\Gamma(U^{i-1}), \quad 
    \partial: \Gamma(U^{i})\longrightarrow\Gamma(U^{i+1}) \,.
\end{equation}
 Thus, we obtain a $\Z$-graded differential complex $\{\Gamma(U^{i}),\Bar{\partial}\}$ and the cohomology of this complex is called the generalized Dolbeault cohomology of $M$, 
\begin{equation}\label{D cohomology}
    GH^{\bullet}_{\Bar{\partial}}(M)=\frac{\ker(\Bar{\partial}: \Gamma(U^{\bullet})\longrightarrow\Gamma(U^{\bullet-1}))}{\img(\Bar{\partial}: \Gamma(U^{\bullet+1})\longrightarrow\Gamma(U^{\bullet}))}\,.
\end{equation}

To get an idea about generalized Dolbeault cohomology, it is useful to consider it first for some simple cases as follows. 

\begin{itemize}
\setlength\itemsep{1em}
    \item[(a)] When $M$ is a complex manifold, The canonical line bundle of the GCS  is just $\wedge^{(n,0)}T^{*}$ and $\Bar{L}= T^{1,0}\oplus (T^{*})^{0,1}$. One can see that 
    $$U^{\bullet}=\oplus_{p-q=\bullet}\wedge^{(p,q)}T^{*}.$$
    So in this case, the generalized Dolbeault cohomology is just 
    \begin{equation}\label{complex}
        GH^{\bullet}_{\Bar{\partial}}(M)=\oplus_{p-q=\bullet}H^{q}(M,\Omega^{p}(M))=\oplus_{p-q=\bullet}H^{p,q}(M).
    \end{equation}
    \item[(b)] When $(M,\omega)$ is a symplectic manifold, the canonical bundle is generated by $e^{i\omega}$ and its null space is $L=\{X-i\omega(X,\,\cdot)|\,X\in T\otimes\C\}$. By \cite[Theorem 2.2]{Cav05}, one can see that $$U^{\bullet}=\{e^{i\omega}(e^{\frac{\Lambda}{2i}}\eta)|\, \eta\in\wedge^{n-\bullet}T^{*}\otimes\C\},$$ where $\Lambda$ is the interior product with the bivector $-\omega^{-1}$. Hence, the generalized Dolbeault cohomology is isomorphic to the complex de Rham cohomology of $M$
    \begin{equation}\label{symplectic}
        GH^{\bullet}_{\Bar{\partial}}(M)=H^{n-\bullet}(M;\C).
    \end{equation}
  \end{itemize}
    Given a GCS  $\mathcal{J}$ on a smooth manifold $M$, we can deform it by a closed real $2$-form $B$, known as \textit{$B$-field transformation}, and get another GCS, 
    \begin{equation}\label{B transformation}
        \mathcal{J}_{B}=
        \begin{pmatrix} 
	       1 & 0 \\
	       -B & 1 \\
	    \end{pmatrix} \mathcal{J}
	    \begin{pmatrix} 
	       1 & 0 \\
	       B & 1 \\
	    \end{pmatrix}.
    \end{equation}
   A local section of the canonical line bundle of  $\mathcal{J}_{B}$ is of the form $e^{B}\wedge\rho$ where $\rho$ is a local section of the canonical line bundle of $\mathcal{J}$. Hence, the canonical line bundle of the deformed structure is 
   $$U_{B}=e^{B}\cdot U \, ,$$
   and the $+i$-eigenbundle of $\mathcal{J}_{B}$ is just $$L_{B}=\{X+\xi-B(X,\,\cdot)\,|\,X+\xi\in L\}.$$
   So, for each $i\in\Z$, we get another decomposition 
   $$\wedge^{\bullet}T^{*}\otimes\C=\bigoplus^{n}_{i=-n}U_{B}^{i},$$ where $$U_{B}^{i}=e^{B}U^{i}.$$ Then for $\beta\in\Gamma (U^{i})$,
   $$d(e^{B}\beta)=e^{B}d\beta=e^{B}\partial\beta+e^{B}\Bar{\partial}\beta,$$
   where $e^{B}\partial\beta\in\Gamma (U_{B}^{i+1})$ and $e^{B}\Bar{\partial}\beta\in\Gamma (U_{B}^{i-1})$.
   Hence, 
   \begin{equation}\label{Bar partial B}
       \Bar{\partial}_{B}=e^{B}\Bar{\partial}e^{-B}
   \end{equation} 
   and 
   \begin{equation}\label{partial B}
      \partial_{B}=e^{B}\partial e^{-B} \,.
   \end{equation}
   The cohomology of the $\Z$-graded complex $\{\Gamma(U_{B}^{i}),\Bar{\partial}_{B}\}$, denoted by $GH_{\Bar{\partial}_{B}}(M)$, is defined as 
   \begin{equation}\label{DB cohomology}
    GH^{\bullet}_{\Bar{\partial}_{B}}(M)=\frac{\ker(\Bar{\partial}_{B}: \Gamma(U^{\bullet}_{B})\longrightarrow\Gamma(U^{\bullet-1}_{B}))}{\img(\Bar{\partial}_{B}: \Gamma(U^{\bullet+1}_{B})\longrightarrow\Gamma(U^{\bullet}_{B}))}\,.
\end{equation}
  Hence, by equation \eqref{Bar partial B}, a $B$-field transformation preserves the generalized Dolbeault cohomology of $M$ up to isomorphism
  \begin{equation}\label{isomorphism}
GH^{\bullet}_{\Bar{\partial}_{B}}(M) \cong GH^{\bullet}_{\Bar{\partial}}(M).
  \end{equation}
   
 A situation of interest to us is that of a  symplectic fiber bundle, that is, a smooth fiber bundle $F\hookrightarrow E\xrightarrow{\pi} M$ with a generic fiber $(F,\sigma)$ which is a compact symplectic manifold, such that the structure group is the group of symplectomorphisms of $(F,\sigma)$.  Angella et al. (see \cite[Theorem2.1, Corollary 2.2]{angella}) study the generalized Dolbeault cohomology of a symplectic fiber bundle  using a Leray spectral sequence under the following assumptions:
  \begin{enumerate}
  \setlength\itemsep{1em}
    \item $M$ is a compact complex manifold.
    \item There is a closed form $\omega$ on $E$ which restricts to the symplectic form $\sigma$ on the generic $F$.
\end{enumerate}
 
 In the sequel, we study the generalized Dolbeault cohomology of torus principal bundles over complex manifolds. The GCS considered on the total space of such a bundle is compatible with the complex and symplectic structures on the base and the fibers. However, the symplectic structure on the fibers may vary, and there may not exist a global closed form specializing to the symplectic forms on the fibers. However, in Section \ref{sseq}, we observe that the generalized Dolbeualt cohomology of the total space may be studied using a spectral sequence following \cite{angella} by virtue of the regular neighborhood theorem, Theorem \ref{thm:productgcs}.

\section{GCS on principal bundles}

The following construction of a generalized complex structure on a smooth principal torus bundle over a complex manifold is mentioned as Example 2.16 in the thesis of Cavalcanti \cite{Cavth}. We present a detailed argument for the convenience of the reader.   

\begin{prop}\label{gencx}
	Let $(E, \pi, M)$ be a smooth  principal $\T^{2l}$-bundle over a complex manifold $M$ with characteristic class of type $(1,1)$. Then, the total space $E$ admits a family of regular GCS with the fibers as leaves of the associated symplectic foliation.  
\end{prop}

\begin{proof} 
	 Consider a connection $(\theta_1, \ldots, \theta_{2l})$ on the principal bundle $E$
	 corresponding to a decomposition $\T^{2l} = \prod_{j=1}^{2l} S^1$ of Lie groups. By the hypothesis, 
  we may choose the connection so that its curvature form is of type $(1,1)$.
  Then, for each $j$ there exists a $2$-form $\chi_j$ of type $(1,1)$ on $M$ such that 
	 \begin{equation}\label{eq1} d \theta_j = \pi^{*} \chi_j \,. \end{equation}
	  Note that
	  \begin{equation}\label{eq2} \omega := \sum_{j=1}^{l} \theta_{2j-1} \wedge \theta_{2j} \end{equation} 
	   is a $\T^{2l}$-invariant $2$-form on $E$ which restricts to an invariant symplectic form on each fiber of $E$.

   Let $\Omega$ be a local  generator of 
	$\wedge^{(n,0)}(T^{*}M \otimes \mathbb{C})$ where $n = \dim_{\mathbb{C}} (M)$. More precisely, if $(z_1, \ldots, z_{n})$ is a system of local holomorphic coordinates on $M$, we may take $$\Omega = dz_1 \wedge \ldots \wedge dz_{n}\,.$$ 	In addition, let $\eta$ be an arbitrary real closed $2$-form on $E$.  Define 
	\begin{equation}\label{eqrho} \rho := e^{ \eta + i \omega } \wedge \pi^{*} \Omega \,. 
		\end{equation}
	  Then, it is clear that  
	   \begin{equation}\label{eqnz} 
	   \pi^{*} \Omega \wedge  \pi^{*}\overline{\Omega} \wedge \omega^{l} \neq 0 \,. 
	   \end{equation} 
	   Moreover, as $d \Omega = 0 $ and $ d \eta= 0 $, we have 
	  \begin{equation}\label{drho} d\rho = e^{\eta} \wedge i \left(e^{i\omega} - \frac{(i\omega)^l}{l!}  \right) d\omega \wedge \pi^{*}\Omega \,. \end{equation}
	   Using \eqref{eq2}, we have
	    $$ d\omega \wedge \pi^{*} \Omega =  \sum_{j=1}^{l}  (\pi^{*} \chi_{2j-1}  \wedge \theta_{2j}  - \theta_{2j-1} \wedge \pi^{*} \chi_{2j}) \wedge \pi^{*}\Omega \,. $$
	 Note that $$\chi_j \wedge \Omega = 0 $$ for each $j$, as $\chi_j$ is of type $(1,1)$ and $\Omega$ is of type $(n,0)$. Hence, 
	 \begin{equation}\label{drz}
	 d\omega \wedge \pi^{*}\Omega = 0 = d \rho \,.
	 \end{equation}	   
	 
By Definition \ref{gcs} (cf. \cite[Theorem 3.38 and Theorem 4.8]{Gua}), \eqref{eqnz} and \eqref{drz} imply that $E$ admits a generalized complex structure whose canonical line bundle is locally generated by $\rho$ (see also \cite[Section 1]{CG}). 
	\end{proof}

Let $K$ be an even-dimensional compact Lie group and let $G$ denote the complexification of $K$. Let $(E_K, \pi, M)$ be a smooth principal $K$-bundle over a complex manifold $M$. 
We say that $(E_K, \pi, M)$ admits a complexification if it can be obtained by a smooth reduction of structure group from a holomorphic principal $G$-bundle $(E_G, \widetilde{\pi}, M)$.

\begin{theorem}\label{cxtion} Let $K$ be an even dimensional compact Lie group and let $G$ denote the complexification of $K$. Let $(E_K, \pi, M)$ be a smooth principal $K$-bundle over a complex manifold $M$, which admits a complexification. Then $E_K$ admits a family of generalized complex structures.  
\end{theorem}

\begin{proof} Let $\T$ be a maximal torus of $K$. Let $B$ be a Borel subgroup of $G$ containing $K$. Then, by \cite[Section 5]{PT} there exists a complex manifold $X = E_G /B$, such that $E_K$ admits the structure of a principal $T$-bundle, $(E_K, \pi', X)$ say, over $X$. Moreover,  this principal $\T$-bundle over $X$ admits a complexification. So, by \cite[Section 4]{PT}, $(E_K, \pi', X)$ admits a $(1,0)$ connection with $(1,1)$ curvature. 
Now, applying Theorem \ref{gencx} to the bundle  $(E_K, \pi', X)$, we conclude that 
$E_K$ admits a family of generalized complex structures. 	 	
\end{proof}
    
    Specific examples of bundles that admit a complexification include the unitary frame bundle associated with a holomorphic vector bundle of even rank over a complex manifold. We refer the reader to \cite[Section 3]{PT} for more examples. The following example was kindly shared with us by Ajay Singh Thakur.
   
   \begin{example}
   	Let $E \rightarrow M$ be a smooth vector bundle of rank $n$ over a complex manifold $M$. Let $\pi:\mathbb P(E) \rightarrow M$ be the associated projective bundle over $M$ with fiber $\mathbb C \mathbb P^{n-1}$, where $\mathbb P(E)$ is the space of lines in $E$. Let $\mathcal L$ be the tautological complex line bundle over $\mathbb P(E)$.  The restriction of $\mathcal L$ to each fiber is the tautological line bundle $\mathcal O(-1)$ over $\mathbb C \mathbb P^{n-1}$. Let $\omega$ be the first Chern class of the dual bundle $\mathcal L^*$. Note that the Fubini-Study metric $\omega_{FS}$ on $\mathbb C \mathbb P^{n-1}$ is first Chern class of the line bundle $\mathcal O(1)$. Therefore,  $\omega$ a closed two form on $\mathbb P(E)$, whose restriction  to each fiber $\mathbb C \mathbb P^{n-1}$ is the symplectic form $\omega_{FS}$. If $\Omega$ be a local  generator of   	$\wedge^{(n,0)}(T^{*}M \otimes \mathbb{C})$, then define \begin{equation} \rho := e^{  i \omega } \wedge \pi^{*} \Omega.
   	\end{equation}
   	As $d\Omega =0 $ and $d\omega =0$, we have $d\rho =0$. Hence, $\mathbb P(E)$ admits a generalized complex structure.
    \end{example}

 \section{Tubular neighborhood of the fiber of a torus bundle}
 
 Let $\theta$ be a \textit{Maurer-Cartan} connection 1-form on $S^1$.
Consider a decomposition 
\begin{equation}\label{product}
  \T^{2l} = \prod_{j=1}^{2l} S^1  
\end{equation}
 of Lie groups.
Let $P_{i}:\T^{2l}= \prod_{j=1}^{2l} S^1\longrightarrow S^1$ be the projection map on $i$-th coordinate for $i=1,\ldots,2l$. Then 
$(P^{*}_1\theta, \ldots, P^{*}_{2l}\theta)$ is the \textit{Maurer-Cartan} connection on $\T^{2l}= \prod_{j=1}^{2l} S^1$. Note that the $2$-form 
\begin{equation}\label{tdomega}
{\omega}_{\T} =\sum^{l}_{j=1}(P^{*}_{2j-1}\theta\wedge P^{*}_{2j}\theta)
\end{equation} 
gives a symplectic form on $\T^{2l}$.\\ 

Let $M\times\T^{2l}\xrightarrow{\pr_2}\T^{2l}$ is the natural projection map. For each $i\in\{1,2,...,2l\}$, define
\begin{equation}\label{tdtheta i}
\tilde{\theta}_{i}=\pr^{*}_2 P^{*}_{i}\theta\,.
\end{equation}
\medskip
Then, $ (\tilde{\theta}_{1},\ldots,\tilde{\theta}_{2l})$ is a {$\T^{2l}$-invariant} connection of the trivial principal $\T^{2l}$-bundle $M\times\T^{2l}\xrightarrow{\pr_1} M$. \\

Now, let $\pi: E \to M$ be a smooth principal $\T^{2l}$-bundle that satisfies the hypothesis of Proposition \ref{gencx}. Let $\{U_{\alpha}\}$ be a locally finite open cover of $M$ such that $E$ admits a local trivialization  
\begin{equation}\label{phialpha}
    \phi_{\alpha}:\pi^{-1}(U_{\alpha})\longrightarrow U_{\alpha}\times \T^{2l}
\end{equation}
over each $U_{\alpha}$. Let 
$i_{\alpha}: U_{\alpha}\times \T^{2l}\hookrightarrow M\times\T^{2l}$ be the natural inclusion map.
Then, $\phi_{\alpha}^{*} i_{\alpha}^{*} \tilde{\theta}_{i}, \,  1 \leq i \leq 2l,$ are 1-forms on $\pi^{-1}(U_{\alpha})$. 
Let $\{ \psi_{\alpha} \}$ be a smooth partition of unity on $M$ subordinate to $\{U_{\alpha}\}$. For each $j\in\{1,\ldots,2l\}$, define an {$S^1$-invariant} 1-form $\theta_j$ on $E$ by
\begin{equation}\label{theta i}
\theta_{j}=\sum_{\alpha}(\psi_{\alpha}\circ\pi)\,\,\phi_{\alpha}^{*}i_{\alpha}^{*}\tilde{\theta}_{j}\,.
\end{equation}
Then,  $\Theta := (\theta_1, \ldots, \theta_{2l}) $ gives a connection on the principal bundle $E$ corresponding to the decomposition \eqref{product} which is invariant under the $\T^{2l}$ action. Define 
\begin{equation}\label{omega}
\omega\, :=\,\sum^1_{j=1}\theta_{2j-1}\wedge\theta_{2j}\,.
\end{equation}

If the curvature of $\Theta$ is of type $(1,1)$, then by Proposition \ref{gencx}, we have  a family of GCS on $E$ with the canonical line bundle $U_{E}$ locally generated by 
\begin{equation}\label{rho}
\rho\,=\,e^{\eta+i\omega}\wedge\pi^{*}(\Omega)
\end{equation}
where $\eta$ is a closed real $2$-form on $E$.
In general, $\rho$ only gives an {\it almost GCS} on $E$, i.e., the integrability condition \eqref{d rho} may not be satisfied.

Let \begin{equation}\label{rho_0} \rho_0\,:=\,e^{i\omega}\wedge\pi^{*}(\Omega)\,. \end{equation}  
For each $b \in M$, let
\begin{equation}\label{i_b}
    i_{b}: \pi^{-1}(b) = E_{b}\longrightarrow E
\end{equation}
be the natural inclusion map.
Consider the $\T^{2l}$-invariant symplectic structure induced by $\omega$ on $E_{b}$, $$\omega_{b}:= i_{b}^{*}\omega\,.$$
Given $b \in U_{\alpha}$, 
consider the map (cf. \eqref{phialpha})  
$$ \phi^{-1}_{\alpha, b} := \phi^{-1}_{\alpha} (b,\cdot) : \T^{2l} \longrightarrow E_b.$$ 
Similarly, using the identification of $\{b\}\times \T^{2l}$ with $\T^{2l}$, denote  by 
\begin{equation}\label{phi alpha b}
    \phi_{\alpha, b}: E_b \longrightarrow \T^{2l} 
\end{equation}
the restriction of the map $\phi_{\alpha}$ to $E_b$.
\noindent 
Consider the family of symplectic forms $\tilde{\omega}_{b}$ on $\T^{2l}$ defined by
\begin{equation}\label{omegatildeb}
  \tilde{\omega}_{b}=(\phi^{-1}_{\alpha, b})^{*}\omega_{b} \,.   
\end{equation}

\begin{lemma}\label{same localization} 
For any $b\in U_{\alpha}\cap U_{\beta}$, $$(\phi^{-1}_{\alpha, b})^{*}\omega_{b}=(\phi^{-1}_{\beta, b})^{*}\omega_{b} \,.$$ \end{lemma}

\begin{proof}
Consider the composition of the maps $E_{b}\xrightarrow{\phi_{\alpha,b}}\T^{2l}\xrightarrow{\phi^{-1}_{\beta, b}} E_{b}$. Note that $\phi^{-1}_{\beta,b}\circ\phi_{\alpha,b}\in\T^{2l}$. Then, we have, 
\begin{align*}
(&\phi^{-1}_{\beta,b}\circ\phi_{\alpha,b})^{*}\omega_{b} = \omega_{b} \quad \quad (\text{as}\,\,\omega_{b}\,\,\text{is}\,\,\T^{2l}\text{-invariant})\\
 \implies &\phi_{\alpha,b}^{*}\circ (\phi^{-1}_{\beta,b})^{*}\omega_{b} = \omega_{b}\,,\\
\implies & (\phi^{-1}_{\beta, b})^{*}\omega_{b}=(\phi^{-1}_{\alpha, b})^{*}\omega_{b}.
\end{align*}
\end{proof}
Thus,  the symplectic form $\tilde{\omega}_{b}$ that defined  on $\T^{2l}$
by \eqref{omegatildeb} does not depend on the choice of local trivialization.
\begin{remark}
Lemma \ref{same localization} does not depend on the choice of connection i.e, for any connection $(\theta_1, \ldots, \theta_{2l})$ on the principal bundle $E$
	 corresponding to a decomposition $\T^{2l} = \prod_{j=1}^{2l} S^1$ of Lie groups and the corresponding $\omega$ as in Proposition \ref{gencx}, we can use the same techniques and get the same result.
\end{remark}

Now by a similar argument as in Lemma \ref{same localization} and as $\tilde{\theta_i}$ is $\T^{2l}$-invariant, 
$$\phi_{\alpha,b}^{*} (\tilde{\theta}_i\mid_W) = \phi_{\beta,b}^{*} (\tilde{\theta}_i\mid_W)\,,$$ for any open set $W \subset U_{\alpha}\cap U_{\beta}$ and 
any $b \in W$.
Therefore, $$ i_{b}^{*} \circ \phi_{\alpha}^{*}(\tilde{\theta}_{i}\mid_W)= 
i_{b}^{*} \circ \phi_{\beta}^{*}(\tilde{\theta}_{i}\mid_W)\,, $$ for any $b \in W \subset U_{\alpha}\cap U_{\beta}$.

Then, it follows from the local finiteness of $\{U_{\alpha}\}$ and \eqref{theta i} that for any $b\in M$ and $i\in\{1,\ldots,2l\}$, there exists a suitable open set $W$ containing $b$ such that
\begin{align*}
i^{*}_{b}\theta_{i}&
=\sum_{\beta}\psi_{\beta}(b) \, i_{b}^{*} \circ \phi_{\beta}^{*}(\tilde{\theta}_{i}\mid_W) \\
&= \sum_{\beta}\psi_{\beta}(b) \, i_{b}^{*} \circ \phi_{\alpha}^{*}(\tilde{\theta}_{i}\mid_W) \\
&=  i_{b}^{*} \circ \phi_{\alpha}^{*}(\tilde{\theta}_{i}\mid_W) 
\end{align*}
for any  $\alpha$ satisfying $b \in U_{\alpha}$. Hence, we have, 
\begin{equation}\label{i star theta}
\begin{array}{ll}
   i^{*}_{b}\theta_{i} & =  i_{b}^{*} \circ \phi_{\alpha}^{*}(\tilde{\theta}_{i}\mid_W)\\
  &  =   i_{b}^{*} \circ \phi_{\alpha}^{*}(\tilde{\theta}_{i}) \\
  & = (\phi_{\alpha}\circ i_{b})^{*}\tilde{\theta}_i
  \end{array}
\end{equation}
for any  $\alpha$ such that $b \in U_{\alpha}$.

Let us explicitly calculate $\tilde{\omega}_{b}$ for every $b\in M$. As in \eqref{omega}, $\omega\,=\,\sum^1_{j=1}\theta_{2j-1}\wedge\theta_{2j}$. We have
\begin{align*}
i^{*}_{b}\omega&=\sum_{j=1}^{l} i^{*}_{b}\theta_{2j-1}\wedge i^{*}_{b}\theta_{2j}\\
&=\sum^l_{j=1}(\phi_{\alpha}\circ i_{b})^{*}\tilde{\theta}_{2j-1}\wedge(\phi_{\alpha}\circ i_{b})^{*}\tilde{\theta}_{2j} \quad (\text{by}\,\,\eqref{i star theta})\\
&=\sum^l_{j=1}(\phi_{\alpha}\circ i_{b})^{*}(\tilde{\theta}_{2j-1}\wedge\tilde{\theta}_{2j}) \\  
&=\sum^l_{j=1}(\pr_2\circ\,\phi_{\alpha}\circ i_{b})^{*}(P^{*}_{2j-1}\theta\wedge P^{*}_{2j}\theta)\\
&=\sum^l_{j=1}(\pr_2\circ\,\tilde{i}_{b}\circ\phi_{\alpha, b})^{*}(P^{*}_{2j-1}\theta\wedge P^{*}_{2j}\theta)\\
&=\sum^l_{j=1}\phi_{\alpha, b}^{*}(P^{*}_{2j-1}\theta\wedge P^{*}_{2j}\theta)\quad  (\text{as}\,\,\pr_2\circ\,\tilde{i}_{b}=\text{Id}_{\T^{2l}},\, \text{see}\,\,\eqref{tilde i})\\
&=\phi_{\alpha,b}^{*} \,{\omega}_{\T}\quad ({\rm see} \, \eqref{tdomega})
\end{align*}
for any  $\alpha$ such that $b \in U_{\alpha}$.
Therefore, for each $b\in M$, 
\begin{equation}\label{omegabsame}
\tilde{\omega}_{b}
=(\phi^{-1}_{\alpha, b})^{*}(i^{*}_{b}\omega)
=(\phi^{-1}_{\alpha, b})^{*}(\phi_{\alpha, b}^{*} {\omega}_{\T})
={\omega}_{\T} \,.
\end{equation}
In other words, $\tilde{\omega}_{b} $ is independent of the choice of $b \in M$.
\begin{lemma}\label{lem:2conns}
Let $(\theta_1, \ldots, \theta_{2l})$ and $(\theta^{'}_1, \ldots, \theta^{'}_{2l})$ be any two connections on $E$ corresponding to a decomposition of $\T^{2l} = \prod_{i=1}^{2l} S^1$ of Lie groups. Then, for each $j\in\{1,\cdots,2l\}$, there exists a $1$-form $\beta_{j}\in\Omega^1(M)$ such that 
 $$\theta_{j}-\theta^{'}_{j}=\pi^{*}\beta_{j} \,. $$   
\end{lemma}
\begin{proof}
  Denote the connections by $\Theta :=(\theta_1, \ldots, \theta_{2l})$ and $\Theta^{'} := (\theta^{'}_1, \ldots, \theta^{'}_{2l})$. For each $x\in M$,  define the value of $\beta_j$ at $x$, by
  \begin{equation}\label{i beta}
  (\beta_{j})_x(v) := (\theta_{j}-\theta^{'}_{j})_y (w)
  \end{equation}
 for any $v\in T_{x}M$, where $\pi(y)=x$ and $d\pi_{y}(w)=v$.
 First, we show that the definition is independent of the choice of 
 $w$ for fixed $v$ and $y$. 
 Let $w, w^{'}\in (d\pi_{y})^{-1}(v)$. Then, $d\pi_{y}(w-w^{'})=0$. So
 $w-w^{'} \in T_y(E_x)$. There exists a vector $W$ in the Lie algebra $\mathfrak{t}$ of $\T^{2l}$ such that the fundamental vector field $W^{\#}$
 of $W$ satisfies 
 $$ W^{\#}(y) = w- w^{'} \,.$$
 It follows that 
 $$ \Theta_y (w- w^{'}) = W = \Theta^{'}_y (w- w^{'} ) \,. $$
 Thus, 
\begin{equation*}\label{w invariance}
(\theta_{j}-\theta^{'}_{j})_{y}(w)=(\theta_{j}-\theta^{'}_{j})_{y}(w^{'}) \,,
\end{equation*}
showing that the definition of $ (\beta_{j})_x(v) $ in \eqref{i beta} is independent of the choice of $w$.

Moreover, as the structure group is abelian, the connections $\Theta$ and $\Theta^{'}$ on $E$ are $\T^{2l}$-invariant. Given
any  $y, y^{'}\in\pi^{-1}(x)$, there exists $g\in\T^{2l}$ such that $y= r_g(y') = y^{'} \cdot g$. 
Let $w^{'}\in T_{y^{'}M}$ such that $(dr_g)_{y^{'}}(w^{'})=w$ and $d\pi_{y}(w)=v$. Then, 
\begin{align*}\label{y invariance} 
(\theta_{j}-\theta^{'}_{j})_{y}(w) 
 &=(\theta_{j}-\theta^{'}_{j})_{gy^{'}}
((dr_g)_{y^{'}}(w^{'})) \\
&=(r_g^{*}(\theta_{j}-\theta^{'}_{j}))_{y^{'}}(w^{'})\\
&=(\theta_{j}-\theta^{'}_{j})_{y^{'}}(w^{'}) \, .
\end{align*}
This proves that the definition of $(\beta_{j})_{x}$ as in \eqref{i beta} is independent of choices of both $y$ and $w$.

Next, we take advantage of the above independence of choices to show that $\beta_j$ is a smooth form.
Let $$f: =\phi_{\alpha}^{-1}\circ\,i \,$$
where $i: U_\alpha \longrightarrow U_{\alpha} \times \T^{2l}$ is the inclusion map
defined by $i (z) := (z, 1)$.
Here, $1$ denotes the identity element of $\T^{2l}$.
Then,  
$d\pi_{f(z)}(df_{z}(v))=v$ for all $z\in U$ and $v\in T_{z}U$. It follows that
$$\beta_{j}=(\theta_{j}-\theta^{'}_{j})\circ\,df \,,$$
completing the proof of the lemma. 
\end{proof}
Let $(\theta_1, \ldots, \theta_{2l})$ be the connection defined in \eqref{theta i}, and let $\Omega$ denote a local $(n,0)$ form on $M$ as in Proposition \ref{gencx}.
Let $(\theta^{'}_1, \ldots, \theta^{'}_{2l})$ be any connection on $E$, corresponding to the same decomposition $\T^{2l} = \prod_{i=1}^{2l} S^1$ of Lie groups, such that $d \rho_0^{\prime} =0$ where
\begin{equation}\label{rho'_0} \rho^{'}_0\,:=\,e^{i\omega^{'}}\wedge\pi^{*}(\Omega)\, 
\quad {\rm and}  \quad \omega^{'} := \sum_{i=1}^{l} \theta^{'}_{2i-1} \wedge \theta^{'}_{2i} \, . \end{equation}

Then, following the Proposition \ref{gencx}, we get a family of GCS on $E$ with the canonical line bundle $U^{'}_{E}$ , locally generated by 
\begin{equation}\label{rho'}
\rho^{'}\,=\,e^{\eta+i\omega^{'}}\wedge\pi^{*}(\Omega)
\end{equation}
where $\eta$ is a closed real $2$-form on $E$.

Fix $\alpha$. Let $\pr_1: U_{\alpha}\times\T^{2l}\longrightarrow U_{\alpha}$ and 
$\pr_2: U_{\alpha}\times\T^{2l}\longrightarrow \T^{2l} $
be the natural projections. Note that $\Pr_1 \circ\, \phi_{\alpha} = \pi$ on 
$E \mid_{U_{\alpha}}= \pi^{-1} (U_{\alpha})$.
\noindent 
On $\pi^{-1}(U_{\alpha})$, we have the GCS given by $\rho^{'}_0|_{\pi^{-1}(U_{\alpha})}= e^{i\omega^{'}}\wedge\pi^{*}\Omega$. Hence, we get a GCS on $U_{\alpha}\times\T^{2l}$ given by 
\begin{equation}\label{gconualpha}
    \widetilde{\rho}_{\alpha} := (\phi^{-1}_{\alpha})^{*}(\rho^{'}_0|_{\pi^{-1}(U_{\alpha})})=e^{i(\phi^{-1}_{\alpha})^{*}\omega^{'}}\wedge(\phi^{-1}_{\alpha})^{*}\pi^{*}\Omega=e^{i(\phi^{-1}_{\alpha})^{*}\omega^{'}}\wedge\pr_1^{*}\Omega \,.
\end{equation}
Consider the following decomposition.
\begin{equation}\label{decomp}
\begin{aligned}
\Omega^{k}_{\C}(U_{\alpha}\times \T^{2l})=\sum_{r+p+q=k}\pr_2^{*}(\Omega^{r}_{\C}(\T^{2l}))\otimes_{C^{\infty}(U_{\alpha}\times \T^{2l},\C)}&\bigg(\pr_1^{*}(\Omega^{p,0}(U_{\alpha}))\\
\otimes_{C^{\infty}(U_{\alpha},\C)}\pr_1^{*}(\Omega^{0,q}(U_{\alpha}))\bigg) &
\end{aligned}
\end{equation}
Accordingly, $i(\phi^{-1}_{\alpha})^{*}\omega^{'}\in\Gamma(\wedge^2 T^{*}_{\C}(U_{\alpha}\times \T^{2l}))$  
 decomposes into six components,
\begin{equation*}
\begin{matrix}
A^{200}\\
A^{110} & A^{101}\\
A^{020} & A^{011} & A^{002}\,.
\end{matrix}
\end{equation*}
Here, the first superscript in $A^{rpq}$ corresponds to the de Rham grading on $\Omega^{\bullet}_{\C}(\T^{2l})$, 
and the last two superscripts correspond to the Dolbeault grading on 
$\Omega^{\bullet}_{\C}(U_{\alpha})$. Furthermore, the exterior derivative decomposes into the sum of three operators
$$d=d_{F}+\partial+\overline{\partial}\,,$$ each of degree $1$ in their respective component of the tri-grading. Note that $d_{F}$ is the fiber-wise exterior derivative. 
 

Denote the imaginary part of $A^{200}$ by $\widehat{\omega}$. In other words,
\begin{equation}\label{tilde omega}
  A^{200} = i\widehat{ \omega} \,. 
\end{equation}
\noindent
Consider the  maps,
\[\begin{tikzcd}[ampersand replacement=\&]
	{\T^{2l}} \&\& {U_{\alpha}\times\T^{2l}} \&\& {\pi^{-1}(U_{\alpha})} 
	\arrow["{\tilde{i}_{b}}", from=1-1, to=1-3]
	\arrow["{\phi^{-1}_{\alpha}}", from=1-3, to=1-5]
\end{tikzcd}\]
where 
\begin{equation}\label{tilde i}
    \tilde{i}_{b}(x) := (b,x) \,.
\end{equation}
Recall the connection form $\theta_{j}$ in \eqref{theta i}. Then, applying Lemma \ref{lem:2conns} and equation \eqref{omegabsame}, 
for each $b\in M$, we get   
	  \begin{align*}
	  &i^{*}_{b}\theta^{'}_{j}= i^{*}_{b}\theta_{j} \quad \forall \,j\\
	  \implies &i^{*}_{b}\omega^{'} = i^{*}_{b}\omega \quad (\omega\,\, \text{as in}\,\,\eqref{omega})\\
	  \implies & (\phi^{-1}_{\alpha, b})^{*}i^{*}_{b}\omega^{'}
	  = (\phi^{-1}_{\alpha, b})^{*}i^{*}_{b}\omega \\
        \implies & (\phi^{-1}_{\alpha, b})^{*}i^{*}_{b}\omega^{'} = \tilde{\omega}_{b} \\
	  \implies & (\phi^{-1}_{\alpha, b})^{*}i^{*}_{b}\omega^{'}= \omega_{\T}\,.
	  \end{align*}
Then, we have,
\begin{align*}
\tilde{i}_{b}^{*}A^{200}&=i(\tilde{i}_{b}^{*}(\phi^{-1}_{\alpha})^{*}\omega^{'}))\\
&=i(\phi^{-1}_{\alpha}\circ\tilde{i}_{b})^{*}\omega^{'}\\
&=i(i_{b}\circ\phi^{-1}_{\alpha,b})^{*}\omega^{'}\\
&=i(\phi^{-1}_{\alpha, b})^{*}i^{*}_{b}\omega^{'}  \\
&=i\omega_{\T}\,.
\end{align*}
\noindent
Hence, by \eqref{tilde omega} we get, 
\begin{equation}\label{hatomegabsame}
   \tilde{i}_{b}^{*}\widehat{\omega} = {\omega}_{\T} \, \,\forall\,b\in U_{\alpha} \,. 
\end{equation}
\noindent 
Consider the GCS $\widetilde{\rho}_{\alpha}$ on $U_{\alpha} \times \T^{2l}$ from \eqref{gconualpha},
$$ e^{i(\phi_{\alpha}^{-1})^{*}\omega^{'} } \wedge \pr_1^{*} \Omega  = 
    e^{\sum A^{rpq}} \wedge \pr_1^{*} \Omega  \,.$$ 
\noindent
Note that only the components $A^{200}\,,\,A^{101}\,\text{and}\,A^{002}$ act non-trivially, via the wedge product,  on $\pr_1^{*}\Omega$ in the expression $e^{i(\phi^{-1}_{\alpha})^{*}\omega^{'}}\wedge\pr_1^{*}\Omega$, as $\Omega$ is a pure  $(n,0)$-type form. Therefore, $\widetilde{\rho}_{\alpha}$ simplifies to
\begin{equation}\label{rhoalpha}  \widetilde{\rho}_{\alpha} = e^{i \widehat{\omega} + A^{101} + A^{002}} \wedge \pr_1^{*} \Omega  \,. \end{equation}
Then by equation \eqref{drz}, we get that $d(i(\phi^{-1}_{\alpha})^{*}\omega^{'})\wedge\pr_1^{*}\Omega=0$ which implies the following four equations
\begin{equation}\label{Eq1}
\overline{\partial} A^{002}=0     
\end{equation}
\begin{equation}\label{Eq2}
 \overline{\partial} A^{101}+d_{F} A^{002}=0  
\end{equation}
\begin{equation}\label{Eq3}
\overline{\partial} A^{200}+d_{F} A^{101}=0    
\end{equation}
\begin{equation}\label{Eq4}
 d_{F} A^{200}=0\,.   
\end{equation} The last equation just states that the pullback of $i(\phi^{-1}_{\alpha})^{*}\omega^{'}$ to any fiber is a closed form, as we already know by \eqref{hatomegabsame}. 
 Moreover, since $\overline{A^{101}}$ and $\overline{A^{002}}$  are of the type $({110})$ and $(020)$, respectively, their wedge products with $\pr_1^{*} \Omega$ vanish. Therefore, in general, the exponent of $e$ in \eqref{rhoalpha}  may be modified to 
$$i \widehat{\omega} + A^{101} + \overline{A^{101}}  + A^{002} +\overline{A^{002}}+\widehat{A} \,,$$ where $\widehat{A}$ is a real $2$-form of type $(011)$.
Therefore, we obtain,
\begin{equation}\label{rhoalpha2}  \widetilde{\rho}_{\alpha} = 
e^{\widehat{B}+i \widehat{\omega}} \wedge \pr_1^{*} \Omega  \,, \end{equation}
where $\widehat{B}=A^{101} + \overline{A^{101}}  + A^{002} +\overline{A^{002}}+\widehat{A} \,.$

\begin{lemma}\label{omegahat} The form $\widehat{\omega} $ is the pullback of the form $\omega_{\T}$ on $\T^{2l}$ under the projection map $\pr_2: U_{\alpha} \times \T^{2l} \longrightarrow \T^{2l} $, i.e. $$ \widehat{\omega} = \pr_2^{*} \omega_{\T} \,. $$
\end{lemma}

\begin{proof}
 Since $\widehat{\omega}$ is of type ${(200)}$, it is of the form $$\widehat{\omega} =\sum_{1 \le j\leq\,l(2l-1)}a_{j}\pr_2^{*}\omega_{j}$$ where $a_{j}\in C^{\infty}(U_{\alpha}\times\T^{2l})$ and
$\{\omega_{j} : 1\leq j \leq l(2l-1) \}$ is a global frame of the trivial bundle of smooth $2$-forms on $\T^{2l}$. For any $b,b^{'}\in U_{\alpha}$, by \eqref{hatomegabsame},
\begin{align*}
&\tilde{i}_{b}^{*}\widehat{\omega} =\tilde{i}_{b^{'}}^{*}\widehat{\omega}\\
\implies \sum_{j}&a_{j}(b,\cdot)\,\omega_{j} =\sum_{j}a_{j}(b^{'},\cdot)\,\omega_{j}\\
\implies \sum_{j}&(a_{j}(b,\cdot)-a_{j}(b^{'},\cdot))\,\omega_{j}=0\\
\implies &a_{j}(b,\cdot)=a_{j}(b^{'},\cdot\,).
\end{align*}
Hence, there exists smooth functions $b_{j}\in C^{\infty}(\T^{2l})$ such that $$a_{j}= \pr_2^{*}b_{j}\,.$$ 
Then,
$$ \widehat{\omega} =\pr_2^{*} \bar{\omega} \quad {\rm where} \quad \bar{\omega} =\sum_{j}b_{j}\omega_{j} \in \Omega^2(\T^{2l}) \,. $$ 
Finally, using \eqref{omegabsame} and the fact that $\pr_2 \circ \tilde{i}_b = \id$,  we have
$$ \omega_{\T} =   \tilde{i}_{b}^{*}\widehat{\omega} = \tilde{i}_{b}^{*} \pr_2^{*} \bar{\omega} = \bar{\omega} \,.$$
\end{proof}
Let $\rho^{'}_1=e^{i \pr^{*}_2 {\omega}_{\T}} \wedge \pr_1^{*} \Omega  \,.$ Then from \eqref{rhoalpha2} and Lemma \ref{omegahat}, we can see that the generalized complex structure $\widetilde{\rho}_{\alpha}$, from \eqref{gconualpha}, is of the form
 \begin{equation}\label{main eq1}
 \widetilde{\rho}_{\alpha}= e^{\widehat{B}}\rho^{'}_1\,\,\,\,\,\,\,\,\,\,\,\,\text{on}\,\,\,U_{\alpha}\times \T^{2l}\,,   
 \end{equation} 
 where $\widehat{B}=A^{101} + \overline{A^{101}}  + A^{002} +\overline{A^{002}}+\widehat{A}$, as defined in \eqref{rhoalpha2}. Now $d\widetilde{\rho}_{\alpha}=0$ because $d\rho^{'}=0$, where $\rho^{'}$ as in \eqref{rho'}. This implies 
 \begin{align*}
&e^{\widehat{B}}\wedge d\widehat{B}\wedge\rho^{'}_1=0\,\,\,\,\,(\text{as}\,\,\,d\rho^{'}_1=0)\\
\implies &d\widehat{B}\wedge\rho^{'}_1=0\\
\implies &d\widehat{B}\wedge\pr_1^{*} \Omega=0\,.
 \end{align*}
 So, to ensure $d\widehat{B}=0$, it is enough to show that $(d\widehat{B})^{012}$ and $(d\widehat{B})^{111}$ both are zero. This imposes the following two constraint equations,
 \begin{equation}\label{Eq5}
 (d\widehat{B})^{012}=\partial A^{002}+\overline{\partial}\widehat{A}=0   
 \end{equation}
 \begin{equation}\label{Eq6}
 (d\widehat{B})^{111}=\partial A^{101}+\overline{\partial A^{101}}+d_{F}\widehat{A}=0\,.   
 \end{equation}
For each $j\in\{1,2,\ldots,2l\}$, set $$\theta^{''}_{j,\alpha}:=(\phi^{-1}_{\alpha})^{*}\theta^{'}_{j}\,.$$ Then $(\theta^{''}_{1,\alpha},\ldots,\theta^{''}_{2l,\alpha})$ defines a $\T^{2l}$-invariant connection on $U_{\alpha}\times\T^{2l}$. Consider the connection $(\tilde{\theta}_{1},\ldots,\tilde{\theta}_{2l})$ on $U_{\alpha}\times\T^{2l}$ as defined in \eqref{tdtheta i}. By Lemma \ref{lem:2conns}, there exist $\beta_{j,\alpha}\in\Omega^1(U_{\alpha})$ such that 
\begin{equation}\label{cmparison eq}
\theta^{''}_{j,\alpha}-\tilde{\theta}_{j}=\pr^{*}_1\beta_{j,\alpha}\,.    
\end{equation}
Then, we have,
\begin{align*}
(\phi^{-1}_{\alpha})^{*}\omega^{'}&=\sum^{l}_{j=1}\theta^{''}_{2j-1,\alpha}\wedge\theta^{''}_{2j,\alpha}\,,\\
&=\pr^{*}_2\omega_{\T}+\sum^{l}_{j=1}(\tilde{\theta}_{2j-1}\wedge\pr^{*}_1\beta_{2j,\alpha}+\pr^{*}_1\beta_{2j-1,\alpha}\wedge\tilde{\theta}_{2j})\,\\
&+\sum^{l}_{j=1}\pr^{*}_1(\beta_{2j-1,\alpha}\wedge\beta_{2j,\alpha})\,.
\end{align*}
Let $\beta^{pq}_{j,\alpha}$ correspond to the Dolbeault grading of $\beta_{j,\alpha}\,,$ on $\Omega_{\C}^{\bullet}(U_{\alpha})\,.$ One can see that 
\begin{equation}\label{A101-A002}
    \begin{aligned}
&A^{002}=\sum^{l}_{j=1}\pr^{*}_1(\beta^{01}_{2j-1,\alpha}\wedge\beta^{01}_{2j,\alpha})\,,\\
&A^{101}=\sum^{l}_{j=1}(\tilde{\theta}_{2j-1}\wedge\pr^{*}_1\beta^{01}_{2j,\alpha}+\pr^{*}_1\beta^{01}_{2j-1,\alpha}\wedge\tilde{\theta}_{2j})\,.
    \end{aligned}
\end{equation}
Set $A^{02}:=\sum^{l}_{j=1}(\beta^{01}_{2j-1,\alpha}\wedge\beta^{01}_{2j,\alpha})\,.$ Then,  
\begin{equation}\label{A002}
  A^{002}=\pr_1^{*}A^{02}\,.  
\end{equation}
 From \eqref{Eq1}, we get $\overline{\partial}A^{02}=0$. By using local $\overline{\partial}$-Poincar\'{e} Lemma on $U_{\alpha}$, there exists a smooth form $\eta$ of type $(01)$ on $U_{\alpha}$ such that 
\begin{equation}\label{A02} A^{02}=\overline{\partial}\eta\,\,\,\,\,\text{on}\,\,\, U_{\alpha}\,.
\end{equation}
Let us assume that $\widehat{A}=\pr_{1}^{*}A^{11}$ where $A^{11}$ is a real form of type $(11)$ on $U_{\alpha}$ . Then, equation \eqref{Eq5} is equivalent to 
\begin{equation}\label{A11-eta}
    \overline{\partial}(A^{11}-\partial\eta)=0\,\,\,\,\,\text{on}\,\,\, U_{\alpha}\,.
\end{equation} Again, by using the local $\overline{\partial}$-Poincar\'{e} Lemma on $U_{\alpha}$, a smooth form $\eta^{'}$ of type $(10)$ on $U_{\alpha}$ such that 
\begin{equation}\label{eta'} A^{11}-\partial\eta=\overline{\partial}\eta^{'}\,\,\,\,\,\text{on}\,\,\, U_{\alpha}\,.
\end{equation} Since $A^{11}$ is real, $A^{11}-\partial\eta-\overline{\partial\eta}$ is both $\partial$ and $\overline{\partial}$ closed form. Then, by the local $\partial\overline{\partial}$-Lemma on $U_{\alpha}$, there exists a smooth function $\chi\in C^{\infty}(U_{\alpha},\R)$ such that, on $U_{\alpha}$
\begin{equation}\label{chi}
\begin{aligned}
 &A^{11}-\partial\eta-\overline{\partial\eta}=i\partial\overline{\partial}\chi\\
 \implies &A^{11}=\partial\eta+\overline{\partial\eta}+i\partial\overline{\partial}\chi\,.
\end{aligned}   
\end{equation} So, we can see that the general solution of equation \eqref{A11-eta} is \eqref{chi}. Thus, for any choice of such a $\chi$, we get a desirable $A^{11}$ as well as $\widehat{A}$ such that the first condition \eqref{Eq5} is satisfied. By \eqref{A101-A002}, we observe that $$\partial A^{101}+\overline{\partial A^{101}}=\sum^{l}_{j=1}\left[\tilde{\theta}_{2j-1}\wedge\pr^{*}_1(\partial \beta_{2j,\alpha}^{01}+\overline{\partial\beta_{2j,\alpha}^{01}})+ \pr^{*}_1(\partial\beta_{2j-1,\alpha}^{01}+\overline{\partial\beta_{2j-1,\alpha}^{01}})\wedge\tilde{\theta}_{2j}  \right]\,.$$ Then, the second equation \eqref{Eq6} is equivalent to 
\begin{equation}\label{equiv eq}
\partial\beta_{j,\alpha}^{01}+\overline{\partial\beta_{j,\alpha}^{01}}=0\,\,\,\,\,\text{for all}\,\,\,j\in\{1,\ldots,2l\}\,.  
\end{equation}
Since $d\tilde{\theta}_{j}=0\,,$ by \eqref{cmparison eq}, the curvature of the connection $(\theta^{''}_{1,\alpha},\ldots,\theta^{''}_{2l,\alpha})$ is
\begin{equation}\label{eq:curvature}
(\pr_1^{*}d\beta_{1,\alpha},\ldots,\pr_1^{*}d\beta_{2l,\alpha})\,.
\end{equation}

\begin{theorem}\label{thm:productgcs}
Let $E$ be a principal $\T^{2l}$-bundle over an $n$-dimensional complex manifold $M$. Let 
$ \Theta^{\prime} := (\theta^{'}_1, \ldots, \theta^{'}_{2l})$ be any connection on $E$ corresponding to a decomposition $\T^{2l} = \prod_{i=1}^{2l} S^1$ of Lie groups. Let $\omega^{'} := \sum_{i=1}^{l} \theta^{'}_{2i-1} \wedge \theta^{'}_{2i} $. 
Let $\Omega$ be a local  generator of $\wedge^{(n,0)}(T^{*}M \otimes \mathbb{C})$
over the trivializing open set $U_{\alpha} \subset M$. Set
\begin{equation*}
	 \rho^{'} := e^{i \omega^{'}} \wedge \pi^{*} \Omega  
		\end{equation*}
Then we have the following
\begin{enumerate}
\setlength\itemsep{1em}
    \item The condition $d\rho^{'}=0$ gives a GCS of type $\dim_{\C}(M)$ if and only if the curvature of the connection $\Theta^{\prime}$ is of type $(1,1)$.
    \item Let $\{U_{\alpha},\phi_{\alpha}\}$ be a local trivialization. Then,  $\rho^{'}$ is equivalent (via $B$-field transformation and diffeomorphism) to the product GCS 
 $$(\phi^{-1}_{\alpha})^{*}(\rho^{'}|_{\pi^{-1}(U_{\alpha})})\cong e^{i\pr^{*}_2 {\omega}_{\T}} \wedge\pr_1^{*}\Omega \,.$$ 
 on every $U_{\alpha}\times\T^{2l}$ if and only if the curvature of the connection $\Theta^{'}$ is trivial.
\end{enumerate}  
\end{theorem}
\begin{proof}
    \begin{enumerate}
    \setlength\itemsep{1em}
        \item The sufficiency direction follows from the proof of Proposition \ref{gencx}.
\medskip
        
        For the other direction, let $\{U_{\alpha},\phi_{\alpha}\}$ be a local trivialization. Consider the connection $(\tilde{\theta}_{1},\ldots,\tilde{\theta}_{2l})$ on $U_{\alpha}\times\T^{2l}$ as defined in \eqref{tdtheta i}. Let $d\rho^{'}=0\,.$\\ Then, on $U_{\alpha}\times\T^{2l}$, we have
        \begin{equation*}
        \begin{aligned}
        &d(\phi^{-1}_{\alpha})^{*}(\rho^{'})=0\\
        \implies &d(e^{\widehat{B}}\rho_{1}^{'})=0\,, \quad \text{where}\; \rho_{1}^{'}=e^{i\pr_{2}^{*}\omega_{\T}}\wedge\pr_{1}^{*}\Omega \\
        \implies &d\widehat{B}\wedge\pr_{1}^{*}\Omega=0\\
        \implies &d(A^{101}+A^{002})\wedge\pr_{1}^{*}\Omega=0\,, \quad \text{as $\Omega$ is of type $(n,0)$}\\
        \implies &\overline{\partial}A^{101}=0\,,  \quad \text{as $d_F A^{002} = 0  $ }\\
        \implies
        &\sum^{l}_{j=1}\left[\tilde{\theta}_{2j-1}\wedge\pr^{*}_1(\overline{\partial} \beta_{2j,\alpha}^{01})+\pr^{*}_1(\overline{\partial} \beta_{2j-1,\alpha}^{01})\wedge\tilde{\theta}_{2j}\right]=0 \,,\quad \text{ see \eqref{A101-A002}}\\
        \implies &\overline{\partial} \beta_{j,\alpha}^{(0,1)} =0 \quad \text{for all $j$} \,.
        \end{aligned}
        \end{equation*}
        This shows that the $(0,2)$  component of the curvature is zero. Since the curvature is real, it follows that the $(2,0)$ component of the curvature is also zero. 
        
        \item  By part (1), the curvature is assumed to be of type $(1,1)$. Then it follows from \eqref{eq:curvature}, that the curvature is of the form 
      $$ (\pr_1^{*}\Omega_{1,\alpha},\ldots,\pr_1^{*}\Omega_{2l,\alpha})\, $$
     where $\Omega_{j,\alpha}=\partial\beta_{j,\alpha}^{01}+\overline{\partial\beta_{j,\alpha}^{01}}\,$ for all $j$.
        Then, by equation \eqref{equiv eq}, the GCS is a product on local trivializations if and only if the curvature is zero.  
    \end{enumerate}
\end{proof}

\begin{theorem}\label{thm:productgcs3}
 Let $E$ be a principal $G$-bundle over an $n$-dimensional complex manifold $M$ with structure group a symplectic manifold $(G,\omega_{G})$. If there exists a GCS, $\rho^{'}\,,$ of type $\dim_{\C}(M)$ such that, on each trivialization $\{U_{\alpha},\phi_{\alpha}\}$, it is equivalent (via $B$-field transformation and diffeomorphism) to the product GCS 
 $$(\phi^{-1}_{\alpha})^{*}(\rho^{'}|_{\pi^{-1}(U_{\alpha})})\cong
 e^{i\pr^{*}_2 {\omega}_{G}} \wedge\pr_1^{*}\Omega \,,$$ then $E$ is a flat symplectic $G$-bundle where $\Omega$ is a local  generator of $\wedge^{(n,0)}(T^{*}M \otimes \mathbb{C})\,.$
\end{theorem}
\begin{proof}
 Consider the map
$$\psi:U_{\alpha\beta}\times G\longrightarrow U_{\alpha\beta}\times G$$ defined by $$\psi(m,f)=(m,\phi_{\alpha\beta}(m)f)\,\,\,\,\,\text{for all $(m,f)\in U_{\alpha\beta}\times G$}\,,$$ where $U_{\alpha\beta}=U_{\alpha}\cap U_{\beta}$ and $\phi_{\alpha\beta}=\phi_{\alpha}\circ\phi^{-1}_{\beta}:U_{\alpha\beta}\longrightarrow G$ is the transition map. By the assumption on the GCS, there exists a real closed form $B_{\alpha}\in\Omega^{2}(U_{\alpha}\times G)\,,$ such that
\begin{equation}\label{prduct gcs}
(\phi^{-1}_{\alpha})^{*}(\rho^{'}|_{\pi^{-1}(U_{\alpha})})= e^{B_{\alpha}+i\pr^{*}_2 {\omega}_{G}} \wedge\pr_1^{*}\Omega \,.    
\end{equation}
On $U_{\alpha\beta}\times G\,,$ we again denote $B_{\alpha}|_{U_{\alpha\beta}\times G}$ and $B_{\beta}|_{U_{\alpha\beta}\times G}$ by $B_{\alpha}$ and $B_{\beta}\,,$ respectively.
Let $(B_{\alpha}+i\omega_{G})^{rpq}$ denote the $(rpq)$-component of $B_{\alpha}+i\omega_{G}$ corresponding to the decomposition of $\Omega^2_{\C}(U_{\alpha\beta}\times G)\,,$ as defined in \eqref{decomp}\,. Similar
meanings are assigned to $(\psi^{*}(B_{\alpha}+i\omega_{G}))^{rpq}\,,$ $(B_{\beta}+i\omega_{G})^{rpq}\,,$ and $(\psi^{*}(B_{\beta}+i\omega_{G}))^{rpq}\,.$
By \eqref{prduct gcs}, on $U_{\alpha\beta}\times G$, we have 
\begin{align*}
&\psi^{*}(e^{B_{\alpha}+i\pr_2^{*}\omega_{G}}\wedge\pr_1^{*}\Omega)= 
e^{B_{\beta}+i\pr_2^{*}\omega_{G}}\wedge\pr_1^{*}\Omega\\
\implies &e^{\psi^{*}(B_{\alpha}+i\pr_2^{*}\omega_{G})}\wedge\pr_1^{*}\Omega= 
e^{B_{\beta}+i\pr_2^{*}\omega_{G}}\wedge\pr_1^{*}\Omega\\
\implies &e^{\sum (\psi^{*}(B_{\alpha}+i\pr_2^{*}\omega_{G}))^{r0q}}\wedge\pr_1^{*}\Omega=
e^{\sum (B_{\beta}+i\pr_2^{*}\omega_{G})^{r0q}}\wedge\pr_1^{*}\Omega\,,\quad\text{as $\Omega$ is of type $(n,0)$}\\
\implies &\sum (\psi^{*}(B_{\alpha}+i\pr_2^{*}\omega_{G}))^{r0q}=\sum (B_{\beta}+i\pr_2^{*}\omega_{\T})^{r0q}\,.
\end{align*}
For $m\in U_{\alpha\beta}$, consider the map $\tilde{i}_{m}$, as defined in \eqref{tilde i}. Then,
\begin{align*}
&\sum \tilde{i}^{*}_{m}(\psi^{*}(B_{\alpha}+i\pr_2^{*}\omega_{G}))^{r0q}=\sum \tilde{i}^{*}_{m}(B_{\beta}+i\pr_2^{*}\omega_{G})^{r0q}\\
\implies &\tilde{i}^{*}_{m}(\psi^{*}(B_{\alpha}+i\pr_2^{*}\omega_{G}))^{200}=\tilde{i}^{*}_{m}(B_{\beta}+i\pr_2^{*}\omega_{G})^{200}\\
\implies &\phi_{\alpha\beta}(m)^{*}(\tilde{i}^{*}_{m}B_{\alpha}+i\omega_{G})=\tilde{i}^{*}_{m}B_{\beta}+i\omega_{G}\\
\implies &\phi_{\alpha\beta}(m)^{*}\omega_{G}=\omega_{G}\,,\quad\text{as $B_{\alpha}\,,$ $B_{\beta}\,$ and $\phi_{\alpha\beta}(m)$ are real\,.}
\end{align*}
This shows that $E$ is a symplectic $G$-bundle. Due to the property that $E$ is a symplectic bundle, one can see that $(\psi^{-1})^{*}$ preserves $e^{i\pr^{*}_2 {\omega}_{G}} \wedge\pr_1^{*}\Omega\,,$ i.e,
$$(\psi^{-1})^{*}\rho_0=\rho_0\,,\quad\text{where $\rho_0=e^{i\pr^{*}_2 {\omega}_{G}} \wedge\pr_1^{*}\Omega$}\,.$$
Let $L$ be the $+i$-eigenbundle (i.e, null space) of the product GCS\,, $\rho_0$\,. At a point $(m,f)\in U_{\alpha\beta}\times G$, $L$ can be written in the following form, 
\begin{equation}\label{L}
L_{(m,f)}=\left(T^{0,1}_{m}U_{\alpha\beta}\oplus (T^{1,0}_m U_{\alpha\beta} )^{*}\right)\oplus\{X-i\omega_{G,f}(X)\,|\,X\in T_{f}G\otimes\C\}\,.   
\end{equation}
For any $X+\eta\in L$, $\psi_{*}(X)+(\psi^{-1})^{*}\eta$ is again an element of $L$ which is verified from the following:  
\begin{align*}
(\psi_{*}(X)+(\psi^{-1})^{*}\eta)\cdot\rho_0&=\rho_0(\psi_{*}(X))+ (\psi^{-1})^{*}\eta\wedge\rho_0\\ &=(\psi^{-1})^{*}\left(\psi^{*}\left(i_{\psi_{*}(X)}(\psi^{-1})^{*}\rho_0\right)\right)+( \psi^{-1})^{*}(\eta\wedge\rho_0)\,,\quad\text{as $(\psi^{-1})^{*}\rho_0=\rho_0$}\\ 
&=(\psi^{-1})^{*}(i_{X}\rho_0+\eta\wedge\rho_0)\\
&=0\,,\quad\text{as $(X+\eta)\cdot\rho_0=0$}\,.
\end{align*}
It follows that 
\begin{equation}\label{compr}
   X+\eta\in L\,\quad\,\text{if and only if}\,\quad\,\psi_{*}(X)+(\psi^{-1})^{*}\eta \in L
\end{equation}
Note that for $(m,f)\in U_{\alpha\beta}\times G$,
\begin{equation}\label{psi}
\begin{aligned}
(\psi)_{*(m,f)}&=
   \begin{pmatrix}
       Id_{U_{\alpha\beta}}   &0\\
       (r_{f})_{*}\circ(\phi_{\alpha\beta})_{*m}    &(\phi_{\alpha\beta}(m))_{*f}
   \end{pmatrix}\,,\\
   \\
   (\psi^{-1})^{*}_{(m,f)}&=
   \begin{pmatrix}
       Id_{U_{\alpha\beta}}   & (\phi^{-1}_{\alpha\beta})^{*}_{m}\circ(r_{\phi_{\alpha\beta}(m)\cdot f})^{*}\\
       0   &(\phi^{-1}_{\alpha\beta}(m))^{*}_{f}
   \end{pmatrix}\,,
\end{aligned}
\end{equation}
where the map $r_{f}:G\longrightarrow G$ is the right translation by $f$.\\
Let $e\in G$ be the identity element and $Y\in T_{e}G \,.$ Then, for $(m,e)\in U_{\alpha\beta}\times G\,,$ we have 
\begin{align*}
 &\begin{pmatrix}
       (\psi)_{*}   &0\\
       0    &(\psi^{-1})^{*}
   \end{pmatrix}(Y-i\omega_{G,e}(Y))\\
   &=(\psi)_{*}(Y)-(\psi^{-1})^{*}(i\omega_{G,e}(Y))\\
   &=\big\{(\phi_{\alpha\beta}(m))_{*e}(Y)-i(\phi^{-1}_{\alpha\beta}(m))^{*}_{e}(\omega_{G,e}(Y))\big\}\,\quad\text{(By \eqref{psi})}\\
   &\quad-i\left((\phi^{-1}_{\alpha\beta})^{*}_{m}\circ(r_{\phi_{\alpha\beta}(m)})^{*}(\omega_{G,e}(Y))\right)\,,
\end{align*}
By \eqref{L} and \eqref{compr}, $\{Y-i\omega_{G,e}(Y)\}\in L_{(m,e)}$ implies that $$(\psi)_{*}(Y)-(\psi^{-1})^{*}(i\omega_{G,e}(Y))\in L_{(m,\phi_{\alpha\beta}(m))}\,.$$ Then, it follows that $$\eta:=-i\left((\phi^{-1}_{\alpha\beta})^{*}_{m}\circ(r_{\phi_{\alpha\beta}(m)})^{*}(\omega_{G,e}(Y))\right)\in (T^{1,0}_{m}U_{\alpha\beta})^{*}\,.$$ 
Since $(\phi^{-1}_{\alpha\beta})^{*}_{m}$ and $(r_{\phi_{\alpha\beta}(m)})^{*}$ both are real linear operators, and $\omega_{G,e}(Y)$ is a real $1$-form, we get $\overline{\eta}=-\eta$. This contradicts the fact that $\eta\in (T^{1,0}_{m}U_{\alpha\beta})^{*}\,,$ and so,
$$(\phi^{-1}_{\alpha\beta})^{*}_{m}\circ(r_{\phi_{\alpha\beta}(m)})^{*}(\omega_{G,e}(Y))=0\,.$$
Now $(r_{\phi_{\alpha\beta}(m)})^{*}$ is an isomorphism and $\omega_{G,e}$ is non-degenerate which shows that $(\phi^{-1}_{\alpha\beta})^{*}_{m}$ will
vanish. As $m\in U_{\alpha\beta}$ is arbitrary, we get
$$(\phi_{\alpha\beta})^{*}_{m}=0\,,\quad\text{for all $m$}\,.$$ Hence $E$ is a flat symplectic principal $G$-bundle.
\end{proof}

\begin{theorem}\label{thm:productgcs2}
 Let $E$ be a principal $\T^{2l}$-bundle over an $n$-dimensional complex manifold $M$. Then, $E$ is a flat symplectic $\T^{2l}$-bundle if and only if there exists a GCS, $\rho^{'}\,,$ of type $\dim_{\C}(M)$ such that, on each trivialization $\{U_{\alpha},\phi_{\alpha}\}$, it is equivalent (via $B$-field transformation and diffeomorphism) to the product GCS 
 $$(\phi^{-1}_{\alpha})^{*}(\rho^{'}|_{\pi^{-1}(U_{\alpha})})\cong
 e^{i\pr^{*}_2 {\omega}_{\T}} \wedge\pr_1^{*}\Omega \,,$$ where $\omega_{\T}$ is a symplectic form on $\T^{2l}$ and $\Omega$ is a local  generator of $\wedge^{(n,0)}(T^{*}M \otimes \mathbb{C})\,.$
\end{theorem}
\begin{proof}
One way is straightforward as one can construct such a GCS by using Theorem \ref{thm:productgcs}. The converse direction follows from Theorem \ref{thm:productgcs3}.
\end{proof}
\begin{remark}
     Theorem \ref{thm:productgcs2} and Theorem \ref{thm:productgcs3}  do not imply that all GCS of type $\dim_{\C}M$ is the product GCS on a trivializing open set. Even in the simplest case, it may happen that there exists a GCS which cannot be the product GCS in a trivializing open neighborhood. The following example demonstrates this.
\end{remark}

\begin{example}
    Let $E=F\times\C$ be the trivial bundle over $\C$ with symplectic fiber $F$ of dimension $2l$. Let $\omega_{F}$ be a symplectic form on $F$. Set 
     $$A:=iz\overline{z}\,dz\,.$$
    Let $\sigma$ be a real closed, but not exact, $1$-form on $F\,.$ Consider the following $2$-form 
    $$i\omega=i\omega_{F}+(A-\overline{A})\wedge\sigma\,.$$ Now $d(A-\overline{A})=-i\,(\overline{z}+z)\,dz\wedge d\overline{z}\,.$ This implies $d\omega\neq 0$ but $d(i\omega)\wedge dz=0\,.$\\ Note that $\omega^{l}\wedge dz\wedge d\overline{z}= \omega_F^{l}\wedge dz\wedge d\overline{z} \neq 0\,.$ Thus, we have a GCS of type $1$ given by $$\rho=e^{i\omega}\wedge dz\,.$$ If possible, let there exists a closed real $B\in\Omega^2(E)$ such that 
$$\rho=e^{B+i\omega_{F}}\wedge dz \,.$$
   Let $C^{rpq}$  denote the $(rpq)$-component of $C:=B+i\omega_{F}$ in the natural decomposition of  $\Omega^{2}_{\C}(F\times\C)\,,$ as given in \eqref{decomp}. Then, we have 
    \begin{align*}
    &\rho=e^{B+i\omega_{F}}\wedge dz\\
        \implies & e^{-\overline{A}\wedge\sigma}\wedge dz=e^{\sum C^{r0q}}\wedge dz\\
        \implies & -\overline{A}\wedge\sigma=C^{101}\,.
    \end{align*} 
Since $C^{011}$ is real, it is of the form $C^{011}=if\,dz\wedge d\overline{z}$ for some $f\in C^{\infty}(\C\times F,\R)\,.$ Let $d_{F}$ be the exterior derivative in fiber direction.
So, 
\begin{align*}
  & dC=0\\
  \implies & (dC)^{111} = 0\\
  \implies &(i\,d_{F}f+  i (z+\overline{z})\,\sigma)\wedge dz\wedge d\overline{z}=0\\
  \implies & d_{F}f+(z+\overline{z})\,\sigma=0\,.
\end{align*}
Fixing any $z\in\C-\{i\R\}$, we have 
$$d_{F}g=\sigma\,, \quad {\rm where} \quad g=-\frac{1}{z+\overline{z}}f(-,z)\in C^{\infty}(F,\R)\,.$$
This contradicts that $\sigma$ is not exact. If possible, let there exists a real automorphism $h$ on $F\times\C$ such that $$h^{*}\rho=e^{B+i\omega_{F}}\wedge dz\,.$$ 
Note that $h^{*}dz=dz\,,$ and $i\,h^{*}\omega\wedge dz=(B+i\omega_{F})\wedge dz\,.$ So, it follows that 
\begin{align*}
&ih^{*}\omega_{F}+h^{*}(-\overline{A}\wedge\sigma)= \sum C^{r0q}\\
\implies & h^{*}(-\overline{A}\wedge\sigma)=C^{101}
\end{align*}
Then, we can continue as before. Thus we conclude that $\rho$ is not equivalent to the product GCS.
\end{example}

\section{A spectral sequence for the generalized Dolbeault cohomology}\label{sseq} 

 A \textit{generalized holomorphic bundle} over a GC manifold $B$ consists of a complex vector bundle $W$ with a Lie algebroid connection
     $$D:\Gamma(\wedge^{i}L^{*}\otimes W)\longrightarrow\Gamma(\wedge^{i+1}L^{*}\otimes W)$$ satisfying $D\circ D=0$ (cf.\cite[Definition 4.27]{Gua})\,. For a generalized holomorphic bundle $(W,D)$, the Lie algebroid cohomology is defined as
     \begin{equation}\label{lie algebroid cohomology}
         H^{\bullet}(L,W)=\frac{\ker(D:\Gamma(\wedge^{\bullet}L^{*}\otimes W)\longrightarrow\Gamma(\wedge^{\bullet+1}L^{*}\otimes W))}{\img(D:\Gamma(\wedge^{\bullet-1}L^{*}\otimes W)\longrightarrow\Gamma(\wedge^{\bullet}L^{*}\otimes W))}.
     \end{equation}
    For any $2n$-dimensional GC manifold $B$ with canonical line bundle $U$, the corresponding involutive maximal isotropic subbundle $L$, and the operator $\Bar{\partial}$ as in equation \eqref{Bar partial} gives a Lie algebroid connection. Thus $\{U,\Bar{\partial}\}$ is a generalized holomorphic bundle over $B$ and also note that
    \begin{equation}\label{2cohom}
    GH^{n-\bullet}_{\Bar{\partial}}(B) = H^{\bullet}(L,U)\,.
    \end{equation}

Now coming back to our situation, let $L$ be the null space of the canonical line bundle of $E$, denoted as $U_{E}$, as in Theorem \ref{thm:productgcs}. On a local trivialization $\{U_{\alpha}\}$, for a local holomorphic coordinate system $(z_1,\cdots,z_{n})\in U_{\alpha}$, assume that, the GCS on $U_{\alpha}\times\T^{2l}$ is
\begin{equation}\label{product GCS}
(\phi^{-1}_{\alpha})^{*}(\rho_0|_{\pi^{-1}(U_{\alpha})})=e^{i\pr^{*}_2\omega_{\T}}\wedge\pr_1^{*}\Omega
\end{equation}
 where $\Omega=dz_1\wedge\cdots\wedge dz_{n}$ and the null space is 
 $$L|_{\pi^{-1}(U_{\alpha})}\,=\,\pr_1^{*}(T^{\,0,1}U_{\alpha}\oplus (T^{\,1,0}U_{\alpha})^{*})\oplus\pr_2^{*}\{X-i\omega_{\T}(X)\,|\,X\in T(\T^{2l})\otimes\C\}.$$
Further, note that,
$$\pr_2^{*}\{X-i\omega_{\T}(X)\,|\,X\in T(T^{2l})\otimes\C\}=\{X-i\pr_2^{*}\omega_{\T}(X)\,|\,X\in \pr_2^{*}T(\T^{2l})\otimes\C\}.$$ Consider the Courant involutive subbundle $S<L$ such that on local trivialization $$S|_{\pi^{-1}(U_{\alpha})}\,=\,\{X-i\pr_2^{*}\omega_{\T}(X)\,|\,X\in \pr_2^{*}T(\T^{2l})\otimes\C\}.$$ Then following \cite[Section 2]{angella}, for any generalized holomorphic bundle $V$ over $E$, the subspaces 
$$F^{p}\Gamma(\wedge^{p+q}L^{*}\otimes V)=\{\phi\in\Gamma(\wedge^{p+q}L^{*}\otimes V)\,|\,\phi(X_1,\cdots,X_{p+q})=0\,\,\text{for}\,\,X_{n_{1}},\cdots,X_{n_{q+1}}\in S\}$$
of $\Gamma(\wedge^{\bullet}L^{*}\otimes V)$ give a bounded decreasing filtration of $\{\Gamma(\wedge^{\bullet}L^{*}\otimes V),D\}$  such that the corresponding spectral sequence $\{E^{\bullet,\bullet}_{r}\}_{r}$ converges to the Lie algebroid cohomology $H^{\bullet}(L,V)$ described in \eqref{lie algebroid cohomology}. By definition, 
\begin{align*}
E^{p,q}_{0}&=\frac{F^{p}\Gamma(\wedge^{p+q}L^{*}\otimes V)}{F^{p+1}\Gamma(\wedge^{p+q}L^{*}\otimes V)}\\
&=\frac{\{\phi\in\Gamma(\wedge^{p+q}L^{*}\otimes V)\,|\,\phi(X_1,\cdots,X_{p+q})=0\,\,\text{for}\,\,X_{n_{1}},\cdots,X_{n_{q+1}}\in S\}}{\{\phi\in\Gamma(\wedge^{p+q}L^{*}\otimes V)\,|\,\phi(X_1,\cdots,X_{p+q})=0\,\,\text{for}\,\,X_{n_{1}},\cdots,X_{n_{q}}\in S\}}.
\end{align*} 
Locally, we have, $$F^{p}\Gamma(\wedge^{p+q}L^{*}\otimes V)=\bigoplus_{p\leq i\leq p+q}\Gamma\left(\wedge^{i}\pr_1^{*}(L^{*}_{M}|_{U_{\alpha}})\otimes\wedge^{p+q-i}S^{*}|_{\pi^{-1}(U_{\alpha})}\right)\otimes_{C^{\infty}(U_{\alpha}\times\T^{2l}, \C)}\Gamma(V).$$ If $V=\pi^{*}(V^{'})$, for a holomorphic vector bundle $V^{'}$ over $M$ and $L_{M}=T^{\,0,1}M\oplus (T^{\,1,0}M)^{*}$, 
then
\begin{align*}
E^{p,q}_{0}
&\cong\Gamma(\wedge^{p}\pr_1^{*}(L^{*}_{M}|_{U_{\alpha}})\otimes\pr_1^{*}(V^{'}|_{U_{\alpha}}))\otimes_{C^{\infty}(U_{\alpha}\times\T^{2l},\C)}\Gamma(\wedge^{q}S^{*}|_{\pi^{-1}(U_{\alpha})})\\
&\cong\Gamma(\pr_1^{-1}(\wedge^{p}(L^{*}_{M}|_{U_{\alpha}})\otimes V^{'}|_{U_{\alpha}}))\otimes_{C^{\infty}(U_{\alpha}, \C)}\Gamma(\wedge^{q}S^{*}|_{\pi^{-1}(U_{\alpha})})\\
&\cong\Gamma(\wedge^{p}(L^{*}_{M}|_{U_{\alpha}})\otimes V^{'}|_{U_{\alpha}})\otimes_{C^{\infty}(U_{\alpha},\C)}\Gamma(\wedge^{q}S^{*}|_{\pi^{-1}(U_{\alpha})}).\,\,\,\,
\end{align*}
The differential $d_0$ on  $E_0^{p,q}$ is given by  $id \otimes d_{S}$ where $d_{S}$ is the differential on the Lie algebroid complex $\Gamma(\wedge^{\bullet}S^{*})$. For $b\in M$ and $E_{b}=\pi^{-1}(b)$, by \cite[Section 9.2]{voisin02}, we get a flat holomorphic vector bundle $\mathcal{H}^{\bullet}=\cup_{b\in M} H^{\bullet}(E_{b},\C)$ over $M$ where $H^{\bullet}(E_{b},\C)$ denotes the $\C$-valued de Rham cohomology of $E_{b}$. Now, consider the Lie algebroid corresponding to the relative tangent bundle $\mathcal{T}$ of the principal bundle $E$,  and the corresponding Lie algebroid cohomology $H^{\bullet}(\mathcal{T})$. Then, by \cite[Chapter I.2.4]{hattori60}, 
\begin{equation}\label{hattori}
H^{\bullet}(\mathcal{T})\cong\Gamma(M,\mathcal{H}^{\bullet})\, .
\end{equation} 
 Since $T(U_{\alpha}\times\T^{2l})=\pr_1^{*}T(U_{\alpha})\oplus\pr_2^{*}T(\T^{2l})$, we have $\mathcal{T}|_{\pi^{-1}(U_{\alpha})}=\pr_2^{*}T(\T^{2l})$. Moreover, as $\omega_{\T}$ is closed, we have a Lie algebroid isomorphism, $$\mathcal{T}|_{\pi^{-1}(U_{\alpha})} \stackrel{\cong}\longrightarrow S|_{\pi^{-1}(U_{\alpha})}, \quad
X\mapsto X-i\pr^{*}_2\omega_{\T}(X)\,.$$ 
Applying \eqref{hattori}, locally we have, 
$$E^{p,q}_1\cong\Gamma(\wedge^{p}(L^{*}_{M}|_{U_{\alpha}})\otimes V^{'}|_{U_{\alpha}})\otimes_{C^{\infty}(U_{\alpha}, \C)}\Gamma(U_{\alpha},\mathcal{H}^{q}|_{U_{\alpha}}) $$ with the differential $d_1=\bar{\partial}_{M}$, the usual Dolbeault operator on $M$. Hence, globally we have, 
$$E^{p,q}_1\cong\Gamma(\wedge^{p}L^{*}_{M}\otimes V^{'}\otimes\mathcal{H}^{q})$$ with the differential $d_1$ being the Lie algebroid connection for the holomorphic bundle $V^{'}\otimes\mathcal{H}^{\bullet}$.
Hence, we obtain, 
$$E^{p,q}_2\cong H^{p}(L_{M},V^{'}\otimes\mathcal{H}^{q}).$$ Thus, we have a description of the generalized Dolbeault cohomology of the total space which extends the description in \cite[Theorem 2.1]{angella}.

\begin{theorem}\label{main2}
 Let $\pi:E\longrightarrow M$ be a fiber bundle over a complex manifold  $M$ of complex dimension $n$ with a symplectic fiber $(F,\omega_{F})\,.$
Assume that there exists $\omega\in\Omega^{2}(E)$ such that 
\begin{enumerate}
\setlength\itemsep{1em}
    \item it defines a generalized complex structure $\mathcal{J}$ on $E$ which is locally of the form $\rho:=\,e^{i\omega}\wedge\pi^{*}(\Omega)\,,$
    \item on each local trivialization $\{U_{\alpha},\phi_{\alpha}\}$, the GCS is equivalent (via $B$-field transformation and diffeomorphism) to the product GCS as in \eqref{product GCS}, i.e
    \begin{equation*}
(\phi^{-1}_{\alpha})^{*}(\rho|_{\pi^{-1}(U_{\alpha})})\cong e^{i\pr^{*}_2\omega_{F}}\wedge\pr_1^{*}\Omega\,.
\end{equation*}
\end{enumerate}
Here, $\Omega$ is a local  generator of $\wedge^{(n,0)}(T^{*}M \otimes \mathbb{C})\,.$ Let $L$ be the $+i$-eigenbundle of $\mathcal{J}$. Let $V$ be a complex vector bundle over $E$ such that $V=\pi^{*}V^{'}$ for a holomorphic vector bundle $V^{'}$ over the complex manifold $M$. Considering $V$ as a generalized holomorphic bundle, there exists a spectral sequence $\{E^{\bullet,\bullet}_{r}\}_{r}$ which converges to $H^{\bullet}(L,V)$ such that $$E^{p,q}_2\cong H^{p}(L_{M},V^{'}\otimes\mathcal{H}^{q})\,.$$ 
 \end{theorem}
 \begin{proof}
 Follows from the preceding description of the generalized Dolbeault cohomology of the total space.
 \end{proof}
 
 Since $U_{E}=\pi^{*}U_{M}$, where $U_{M}$ is the canonical line bundle for $M$, we have the following theorem.
 \begin{theorem}\label{main3}
 Consider the same setting as in the preceding theorem and $\dim_{\R} F= 2l$. Then there exists a spectral sequence $\{E^{\bullet,\bullet}_{r}\}_{r}$ which converges to $GH^{n+l-\bullet}_{\Bar{\partial}}(E)$ such that $$E^{p,q}_{2}\cong GH^{n-p}_{\Bar{\partial}}(M,\mathcal{H}^{l-q}).$$
 \end{theorem}
\begin{corollary}\label{main4}
For a flat $\T^{2l}$-principal bundle $E$ with the family of GCS as defined in Proposition \ref{gencx}, there exists a spectral sequence $\{E^{\bullet,\bullet}_{r}\}_{r}$ which converges to $GH^{n+l-\bullet}_{\Bar{\partial}}(E)$ such that $$E^{p,q}_{2}\cong GH^{n-p}_{\Bar{\partial}}(M,\mathcal{H}^{l-q}).$$
\end{corollary}
\begin{example}(Generalized Dolbeault cohomology of trivial torus bundles)
When the torus bundle is trivial $E=M\times\T^{2l}$, the flat holomorphic vector bundle $\mathcal{H}^{\bullet}$ is also trivial i.e, $\mathcal{H}^{\bullet}=M\times H^{\bullet}(\T^{2l},\C)$. So the $E^{p,q}_{2}$ term of the spectral sequence as in Theorem \ref{main2} is of the form $$E^{p,q}_2\cong H^{p}(L_{M},U_{M}\otimes H^{l-q}(\T^{2l},\C))=GH^{q}_{\bar{\partial}}(\T^{2l})\otimes GH^{n-p}_{\bar{\partial}}(M).$$
Now, each element in $E_{2}$ term is already a global form on $M\times\T^{2l}$. Hence, $d_{k}$ vanishes for any $k\geq 2\,,$ and $E_2=E_{\infty}$. Therefore, we get the following analogue of the K\"{u}nneth formula

\begin{corollary}\label{main5} 
For the family of generalized complex structures as defined in Proposition \ref{gencx}, when $E=M\times\T^{2l}$, the generalized Dolbeault cohomology group of $E$ has a decomposition in terms of the generalized Dolbeault cohomology groups of the fiber space and the base manifold, i.e,  $$GH^{n+l-m}_{\bar{\partial}}(E)\,\cong\,\bigoplus_{p+q=m}\left(GH^{q}_{\bar{\partial}}(\T^{2l})\otimes GH^{n-p}_{\bar{\partial}}(M)\right)$$
where $-l\leq q\leq l$ , $-n\leq p\leq n$ and $-(n+l)\leq m\leq(n+l)$.
\end{corollary}
\end{example}
 \begin{remark}
In Theorem \ref{main2}, if the form $\omega$ is closed, one may construct a $B$-transformation so that the GCS is the product GCS on each trivializing neighborhood. But even if $\omega$ is not closed, it may still be possible to construct such a $B$-transformation. The following example will show such a construction in the simplest case.      
 \end{remark}  
 \begin{example}
Let $E=\C\times F$ be the trivial bundle over $\C$ with symplectic fiber $F$. Let $\omega_{F}$ is a symplectic form on $F$. Set
\begin{enumerate}
\setlength\itemsep{1em}
    \item $A_1:=(\frac{z^{2}}{2}+z\overline{z})\,d\overline{z}\,.$
    \item $A_2:=z\,d\overline{z}\,.$
\end{enumerate}
Let $\sigma$ be a real closed $1$-form on $F$. Define 
$$i\omega_{j}=i\omega_{F}+(A_{j}-\overline{A_{j}})\wedge\sigma\,,\quad\text{for $j=1,2$}\,.$$ Not that $d(A_1-\overline{A_1})=2(z+\overline{z})\,dz\wedge d\overline{z}$ and $d(A_2-\overline{A_2})=2dz\wedge d\overline{z}$. This implies that $$d\omega_{j}\neq 0\,\quad\text{and}\,\quad d(i\omega_{j})\wedge dz=0\,.$$One can see that $\omega^{l}_{j}\wedge dz\wedge d\overline{z}=\omega_{F}^{l}\wedge dz\wedge d\overline{z}\neq 0$ which implies that 
$$\rho_{j}=e^{i\omega_{j}}\wedge dz\,\,\,\,\,\text{for $j=1,2$}\,,$$ gives a GCS on $E$. One may write 
$$\rho_{j}=e^{B_{j}+i\omega_{F}}\wedge dz\,,\quad\text{where $B_{j}=(A_{j}+\overline{A_{j}})\wedge\sigma$ is a real $2$-form}\,.$$ Notice that $dA_{j}=-d\overline{A_{j}}\,.$ This shows that $dB_{j}=0$ for $j=1,2$. Hence each $\rho_{j}$ is equivalent to the product GCS.
 \end{example}
{\bf Acknowledgement.} The authors would like to thank Ajay Singh Thakur for many stimulating and helpful discussions and for sharing Example 3.3 with them. They also thank an anonymous referee for pointing out a serious error in an earlier draft which helped to improve the manuscript.
The research of the first-named author is supported by a CSIR-UGC NET research grant. The research of the second-named author is supported in part by a SERB MATRICS research grant, 
 MTR/2019/001613.

\end{document}